\documentclass[a4paper,10pt]{article}
\usepackage[utf8]{inputenc}
\usepackage[T1]{fontenc}
\usepackage[english]{babel}
\usepackage{amsmath}
\usepackage{amsthm}
\usepackage{amsfonts}
\usepackage{amssymb}
\usepackage{listings}
\usepackage{graphicx}
\usepackage{mathrsfs}
\usepackage{hyperref}
\usepackage{algpseudocode}
\usepackage{algorithmicx}
\usepackage{algorithm}
\usepackage{subfig}
\usepackage{stmaryrd}
\usepackage{bm}
\usepackage{tikz}
\usepackage{braket}
\usepackage{wasysym}
\usepackage{wrapfig}
\usepackage[framemethod=TikZ]{mdframed}
\usepackage{xcolor}
\usepackage{pstricks}
\usepackage{faktor}
\usepackage[top=2.00cm,bottom=3.00cm,left=2.00cm,right=2.00cm]{geometry} 
\usepackage{todonotes}

\newcommand{\f}{\mathbb}

\newcommand{\ol}{\overline}
\newcommand{\cu}{\subseteq}

\newcommand{\serie}[1]{\{#1_{n}\}_n}

\newcommand{\ve}{\varepsilon}

\theoremstyle{definition}

\theoremstyle{plain}
\newtheorem{theorem}{Theorem}

\newtheorem{lemma}{Lemma}
\newtheorem{corollary}{Corollary}
\theoremstyle{remark}
\newtheorem{example}{Example}

\title{Notes on asymptotic eigenvalues distribution on complex circles}
\author{Giovanni Barbarino}

\begin{document}

\maketitle

\section{Preliminary Results and Notations}

A \textit{Matrix-Sequence} $ \serie A $ is an ordered collection of complex matrices such that $ A_n\in \mathbb C^{n\times n} $. 
We will denote by $\mathscr E$ the space of all matrix-sequences, 
\[
\mathscr  E := \{\serie{A} : A_n\in\mathbb C^{n\times n} \}.
\]

It is often observed in practice that matrix-sequences $ \serie A $  arising from the numerical discretization of  linear differential equations possess a \textit{Spectral Symbol} (see \cite{GLT-book} and references therein), that is, a measurable function describing the asymptotic distribution of the eigenvalues of $A_n$  in the Weyl sense \cite{BS,GLT-book,Tilliloc}. We recall that a spectral symbol associated with a sequence $\serie A$  is a measurable function $f:D\cu \mathbb R^q\to \mathbb C$, $q\ge 1$, satisfying 
\[
\lim_{n\to\infty} \frac{1}{n} \sum_{i=1}^{n} F(\lambda_i(A_n)) = \frac{1}{l(D)}\int_D F(f(x)) dx
\]
for every continuous function $F:\mathbb C\to \mathbb C$ with compact support, where $D$ is a measurable set with finite Lebesgue measure $l(D)>0$ and $\lambda_i(A_n)$ are the eigenvalues of $A_n$. In this case we write 
\[ \serie A\sim_\lambda f. \]

Another important and linked concept is the notion of \textit{Spectral Measure}. Given a matrix sequence $\serie A$ and a measure $\mu$ on $\f C$, we say that $\mu$ is the spectral measure associated to $\serie A$ if
\[
\lim_{n\to\infty} \frac{1}{n} \sum_{i=1}^{n} F(\lambda_i(A_n)) = \int_{\f C} F(x) d\mu\qquad \forall F\in C_c(\f C)
\]
and in this case, we can write $\serie A\sim_\lambda \mu$. In \cite{Radon} we showed that the space of sequences that admit a spectral symbol coincide with the space of sequences admitting a spectral measure that is also a probability measure. In particular, since $(\f C,\mu)$ is a Standard Probability Space, the following lemma can be proved.
\begin{lemma}\label{specmea}
	Given any probability measure $\mu$ on $\f C$, there exists a measurable function $k:[0,1]\to \f C$ such that
	\[
	\int_0^1 F(k(x)) dx= \int_{\f C} F(x) d\mu
	\]
	for every $F\in C_c(\f C)$. In particular, the essential range of $k$ is contained in the support of the measure $\mu$, and if $\mu$ is supported on $\f R$ then $k(x)$ can be chosen monotone increasing.
\end{lemma}
We want to give some result of existence for the spectral symbols, given particular hypothesis on the eigenvalues of $A_n$.  Here we report two theorems that give necessary and sufficient conditions on the existence of a measure given its moments (also known as \textit{moments problem}).

\begin{theorem}[Hausdorff]\label{Hauss}\cite{moment}
	Given a sequence of real numbers $\{m_n\}_{n}$, let 
	\[
	(\Delta m)_n := m_{n+1} - m_n.
	\]
	We say that the sequence $\{m_n\}_{n}$ is a \textit{Completely Monotonic Sequence} if 
	\[
	(-1)^k(\Delta^k m)_n\ge 0\qquad\forall n,k\in\f N.\]
	$\{m_n\}_{n}$ is a completely monotonic sequence if and only if there exists a measure $\mu$ over $[0,1]$ satisfying
	\[
	m_n = \int_0^1 x^n d\mu \qquad \forall n\in\f N.
	\]
	In this case the determined measure is unique.
\end{theorem}

In order to state the second theorem, we have to recall that a Laurent polynomial is an object in $\f C[z,z^{-1}]$, that is also denoted as $\f C[\f Z]$. This space is a unit $*$-algebra with 
\[
q(z) = \sum_{i=-M}^M a_iz^i\in \f C[\f Z] \implies
q^*(z) = \sum_{i=-M}^M \ol a_iz^{-i}\in \f C[\f Z].
\]
 Given a sequence of complex numbers $s = (s_i)_{i\in \f N}$ and its prolongation to $\f Z$ by $s_{-i}:= \ol s_i$, we can define the linear operator
 \[
q(z) = \sum_{i=-M}^M a_iz^i \implies L_s(q) = \sum_{i=-M}^M a_is_{i}.
\]
\begin{theorem}[Carathéodory–Toeplitz]\label{moments_problem_unit_circle}\cite{moment2}
Given a complex sequence $(s_i)_{i\in \f N}$, the following are equivalent:
\begin{enumerate}
	\item There exists a Radon measure on $\f T$ such that
	\[
	s_k  = \int_{\f T} z^{k} d\mu(z) \qquad \forall k\in\f N.
	\]
	\item $L_s(q^*q)\ge 0$ for all $q\in \f C[\f Z]$.
	\item The infinite Toeplitz matrix $H(s) = (s_{j-i})_{i,j=0}^\infty$ is positive semidefinite.
\end{enumerate}

\end{theorem}

Eventually, we have to recall a pair of classical approximation results that let us approximate  continuous functions over compact sets with polynomials or trigonometric polynomials. 

\begin{theorem}[Stone-Weierstrass]\label{Weierstrass}
	$ $
	\begin{enumerate}
\item 
Suppose $f$ is a continuous real-valued function defined on the real interval $[a,b]$. For every $\ve>0$, there exists a polynomial $p$ such that $\|f-p\|_\infty \le\ve$.
\item 
Suppose $f$ is a continuous complex-valued function defined on the unitary circle $\f T$. For every $\ve>0$, there exists a Laurent polynomial $q$ such that $\|f-q\|_\infty \le\ve$.
	\end{enumerate}
\end{theorem}

\begin{theorem}\label{Fou}
	If $f$ is an absolutely continuous real-valued function and $2\pi$-periodic, then the Fourier series of $f$ converges uniformly.
\end{theorem}

\section{Polynomial Limits}

From now on, we will work with double indexed sequences $\{x_{i,n}\}_{i,n}$ of real or complex values, where $i,n\in \f N$, $0<i<n$. Usually, we will deal with convergent sequences 
\[
\lim_{n\to\infty}\frac 1n\sum_{i=1}^n F_k(x_{i,n})\in \f C, \quad \forall k \in\f N
\]
for specific class of functions $F_k$ and we want to know if this is enough to conclude that the limit is described asymptotically by a measure, that is
\[
\lim_{n\to\infty}\frac 1n\sum_{i=1}^n F(x_{i,n}) = \int_{\f C} F(x) d\mu 
\]
for every $F\in C_c(\f C)$. In particular, we want to know if it holds when $F(x)$ is locally a polynomial, so that we can use both Weierstrass and Hausdorff theorems.

\begin{lemma}\label{specsym}
	Let $\{x_{i,n}\}_{i,n}$ be a sequence of real numbers with $1\le i\le n$ and 
	\begin{itemize}
		\item $|x_{i,n}| \le M$ for every $i,n$,
		\item $	h_k = \lim_{n\to\infty}\frac 1n \sum_{i=1}^n x_{i,n}^{k}\in \f R\quad \forall k\in \f N$.
	\end{itemize}
	Then there exists a function $k:[0,1]\to [-M,M]$ such that
	\[
	\lim_{n\to\infty} \frac 1n\sum_{i=1}^n F(x_{i,n}) = \int_0^1 F(k(x))dx
	\]
	for every $F\in C[-M,M]$.
\end{lemma}
\begin{proof}
	Let $y_{i,n}$ be a sequence of real numbers belonging to the interval $[0,1]$ such that 
	\[
	y_{i,n} = \frac{x_{i,n}+M}{2M}.
	\]
A simple computation shows us that
	\[
	\frac 1n \sum_{i=1}^n y_{i,n}^k  = \frac{1}{(2M)^k} \sum_{j=0}^k \binom kj M^{k-j} \frac 1n \sum_{i=1}^n x_{i,n}^j \xrightarrow{n\to\infty} \frac{1}{(2M)^k} \sum_{j=0}^k \binom kj M^{k-j} h_j =: l_k\in \f R.
	\]
	Moreover, $\{l_k\}_k$ is a completely monotonic sequence, since $0\le y_{i,n}\le 1$ for all $i,n$, and 
	\[
	(-1)^s(\Delta^s l)_k = \lim_{n\to \infty} \frac 1n \sum_{i=1}^{n} y_{i,n}^k(1-y_{i,n})^s \ge 0\qquad \forall s,k\in \f N.
	\]
	Owing to Hausdorff Theorem, there exists a measure $\mu$ on $[0,1]$ such that
	\[
	\lim_{n\to \infty}\frac 1n \sum_{i=1}^n y_{i,n}^k = l_k = \int_0^1 y^k d\mu, \qquad \forall k\in \f N.
	\]
	Every $G\in C[0,1]$ can be approximated uniformly by polynomials thanks to Theorem \ref{Weierstrass}, so we conclude that  
	\[
	\lim_{n\to \infty} \frac 1n \sum_{i=1}^{n} G(y_{i,n}) =\int_{0}^1 G(y) d\mu
	\]
 holds for every $G\in C_c(\f C)$. Given now $F\in C[-M,M]$, let $G(y):= F(2My-M)$, that is  a continuous function on $[0,1]$. 
 \[
 \lim_{n\to \infty} \frac 1n \sum_{i=1}^{n} F(x_{i,n}) =\lim_{n\to \infty} \frac 1n \sum_{i=1}^{n} G(y_{i,n}) = \int_{0}^1 G(y) d\mu(y) =
 \int_{-M}^M  F(x) d\mu\left( \frac{x+M}{2M} \right),
 \]
 The measure $\mu$ has finite mass $l_0=1$, so it is a probability measure on $[0,1]$, and consequently $\mu\left( \frac{x+M}{2M} \right)$ is a probability measure on $[-M,M]$. By Lemma \ref{specmea} there exists a monotonic increasing and nonnegative function $k:[0,1]\to [-M,M]$ such that 
 \[
  \lim_{n\to \infty} \frac 1n \sum_{i=1}^{n} F(x_{i,n}) = \int_{-M}^M  F(x) d\mu\left( \frac{x+M}{2M} \right) = \int_{0}^1 F(k(x)) dx.
 \]
\end{proof}

A corollary of Lemma \ref{specsym} allows us to find conditions on the matrix sequences for the existence of a spectral symbol. 
Similar questions have already been answered in \cite{TilliOriginal} and \cite{Donatelli}, where the authors explored minimal conditions for a given function $k$ to be the spectral symbol of a given matrix sequence, using the  test functions $x^k$ for the ergodic formula defining the spectral symbol.
The corollary makes use of the same test functions to derive the existence of a spectral symbol for sequences of Hermitian matrices with bounded norm.

\begin{corollary}
	Given a sequence $\serie A$ of Hermitian matrices with $\|A_n\|$ uniformly bounded by $M$, the sequence
	$\serie A$ admits a spectral symbol if and only if 
	\[
	\lim_{n\to\infty}\frac 1n Tr(A_n^k) \in \f R \qquad \forall k\in \f N.
	\]
\end{corollary}
\begin{proof}
	Given $x_{i,n}$ the eigenvalues of $A_n$, notice that $|x_{i,n}|\le M$ and $\frac 1n Tr(A_n^k) = \frac 1n \sum_{i=1}^{n} x_{i,n}^k$. Owing to Lemma 	\ref{specsym}, the eigenvalues if the sequence $\serie A$ respects
	\[
	\lim_{n\to\infty}\frac 1n Tr(A_n^k) \in \f R \qquad \forall k\in \f N,
	\]
    then there exists a function $h:[0,1]\to \f R$ such that
    \[
    \lim_{n\to\infty} \frac 1n\sum_{i=1}^n F(x_{i,n}) = \int_0^1 F(h(x))dx
    \]
    for every $F\in C_{c}(\f C)$, that coincides with the definition of spectral symbol for $\serie A$.\\\\    
    Vice versa, if we suppose that $\serie A$ admits a spectral symbol $k$, then
    \[
    \lim_{n\to\infty}\frac 1n \sum_{i=1}^{n} x_{i,n}^h = \frac 1{|D|}\int_D k(x)^hdx\in \f R, \qquad \forall h \in \f N.    \]

\end{proof}

Let us now consider $\lambda_{i,n}$ eigenvalues of $A_n$ on a complex circle, meaning that they have all the same norm. There's a simple way to prove that if the averages of $\lambda_{i,n}^k$ converge, then $\serie A$ admit a spectral measure that is also absolutely continuous with respect to the Lebesgue measure, but only if we assume a decay assumption.

\begin{lemma}\label{abs_continuous_spectral_measure}
	For all $n\in\f N$ consider a set of complex numbers $\{\lambda_{i,n}\}_{i=1:n}$ with polar  representations $\lambda_{i,n} = ce^{\mathrm{i}x_{i,n}}$  such that $x_{i,n}\in [-\pi,\pi]$ and $c\in\f R^+$ does not depend on $i,n$. Suppose moreover that
	\[
	s_k:= \lim_{n\to \infty}\frac{1}{n} \sum_{i=1}^n \lambda_{i,n}^k, \quad \forall k\in \f Z 
	\]
and define the function
\[
M(x) := \sum_{k=-\infty}^{\infty} \frac{s_k}{c^k}e^{-\mathrm ikx}.
\]
If  $\sum_{k=-\infty}^\infty \frac{|s_k|}{c^k} <\infty$, then for every function $F\in C[-\pi,\pi]$ the following ergodic relation holds:
\[
\lim_{n\to\infty} \frac{1}{n}\sum_{i=1}^{n} F(x_{i,n})  = \frac 1{2\pi}\int_{-\pi}^{\pi} F(x)M(x) dx.
\]
\end{lemma}
\begin{proof}
First, we notice that $M(x)$ is well defined and continuous, since the series 
\[
\sum_{k=-\infty}^{\infty} \frac{s_k}{c^k}e^{ikx}
\]
converges uniformly.
The trigonometric polynomials are dense in the space of periodic continuous functions on $[-\pi,\pi]$, so we first test the ergodic relation on functions of the type $e^{\mathrm ikx}$.
\[
\frac{1}{n}\sum_{i=1}^{n} \exp(\mathrm ikx_{i,n}) = \frac{1}{c^kn} \sum_{i=1}^n \lambda_{i,n}^k \xrightarrow{n\to\infty} \frac{s_k}{c^k}
\]
\[
\frac 1{2\pi}\int_{-\pi}^{\pi} \exp(\mathrm ikx)M(x) dx =\frac 1{2\pi} \int_{-\pi}^{\pi} \frac{s_k}{ c^k} dx = \frac{s_k}{c^k}. 
\]
The ergodic relation thus holds for every $F\in C_{per}[-\pi,\pi]$ Let now $G$ be any continuous function on $[\pi,\pi]$ and $F$ a continuous periodic function on the same domain that coincides with $G$ on $[-\pi+\ve,\pi-\ve]$ and such that $\|F\|_\infty \le \|G\|_\infty$. 
\begin{align*}
\left|\frac{1}{n}\sum_{i=1}^{n} G(x_{i,n})  - \frac 1{2\pi}\int_{-\pi}^{\pi} G(x)M(x) dx\right| &\le 
\Bigg|\frac{1}{n}\sum_{i=1}^{n} G(x_{i,n}) - \frac{1}{n}\sum_{i=1}^{n} F(x_{i,n})\Bigg|\\
+ \Bigg| \frac{1}{n}\sum_{i=1}^{n} F(x_{i,n}) &-  \frac 1{2\pi}\int_{-\pi}^{\pi} F(x)M(x) dx\Bigg|
 + \left| \frac 1{2\pi}\int_{-\pi}^{\pi} (F(x)-G(x))M(x) dx\right|\\
 &\le 2\|G\|_\infty \frac{\# \set{i | ||x_{i,n}|-\pi |\le \ve  }   }{n}
 + o_n(1) + 2\|G\|_\infty\|M\|_\infty \frac{\ve}{\pi} 
\end{align*}
All that is left to show is that the first term is proportional to $\ve$. Let $H$ be a periodic continuous nonnegative function on $[-\pi,\pi]$ such that $H=1$ on $[-\pi,-\pi+\ve]\cup [\pi-\ve,\pi]$, $H=0$ on $[-\pi+2\ve,\pi-2\ve]$, and $\|H\|_\infty = 1$. we know that
\begin{align*}
\limsup_{n\to\infty}   \frac{\# \set{i | ||x_{i,n}|-\pi |\le \ve  }   }{n}  \le 
\lim_{n\to\infty} \frac{1}{n}\sum_{i=1}^{n} H(x_{i,n})  = \frac 1{2\pi}\int_{-\pi}^{\pi} H(x)M(x) dx
\le \|M\|_\infty \frac{2\ve}{\pi}.
\end{align*} 
The last relation shows that definitively on $n$,
\[
\left|\frac{1}{n}\sum_{i=1}^{n} G(x_{i,n})  - \frac 1{2\pi}\int_{-\pi}^{\pi} G(x)M(x) dx\right| \le
 2\|G\|_\infty \left(\|M\|_\infty \frac{2\ve}{\pi} +\ve \right)
 + o_n(1) + 2\|G\|_\infty\|M\|_\infty \frac{\ve}{\pi}
\]
holds for every $\ve$, leading to the wanted conclusion.
\end{proof}

In the case when $\sum_{k=-\infty}^\infty \frac{|s_k|}{c^k} =\infty$, there may not be an absolutely continuous spectral measure, but 
 $\serie A$ possesses a probability spectral measure anyway, and thus a spectral symbol.


\begin{lemma}\label{spectral_symbol_unit_circle}
	Suppose $\serie A$ is a sequence of matrices $A_n\in \f C^{n\times n}$ with eigenvalues $\{\lambda_{i,n}\}_{i=1:n}$. Suppose that
	\begin{itemize}
		\item $|\lambda_{i,n}| = c$ is independent from $i,n$,
		\item $s_k= \lim_{n\to \infty}\frac{1}{n} \sum_{i=1}^n \lambda_{i,n}^k\in \f C \quad \forall k\in \f Z $.
	\end{itemize}
	In this case, $\serie A$ admits a spectral symbol.
\end{lemma}
\begin{proof}
	Let $\omega_{i,n}:= \lambda_{i,n}/c$ be complex numbers of unitary norm. If $d_k:= s_k/c^k$, then
	\[
	d_k= \lim_{n\to \infty}\frac{1}{n} \sum_{i=1}^n \omega_{i,n}^k\in \f C, \quad \forall k\in \f Z
	\]
	and moreover $d_{-k} = \ol d_k$. Let $\mu_n$ be  the atomic probability measure on $\f T$ induced by $\omega_{i,n}$ and call $r_{k,n}$ its moments
	\[
	\mu_n := \frac{1}{n} \sum_{i=1}^n \delta_{\omega_{i,n}},\qquad
	r_{k,n} = \int_{\f T} x^k d\mu_n = \frac{1}{n} \sum_{i=1}^n \omega_{i,n}^k.
	\]
	The sequence $r_n:=(r_{k,n})_k$ represents the moments of a measure on the unit circle, so we can apply theorem \ref{moments_problem_unit_circle} and obtain that for every $q\in \f C[\f Z]$, the quantity $L_{r_n}(q^*q)$ is nonnegative. We know that $r_{k,n}$ converges to $d_k$ if $n$ goes to infinity, so we can prove that also $L_{r_n}$ converges punctually to $L_d$.
	\[
	q^*q = \sum_{i=-M}^M a_iz^i \implies 
	L_{r_n}(q^*q) = \sum_{i=-M}^M a_ir_{i,n}\xrightarrow{n\to \infty} \sum_{i=-M}^M a_id_{i} = L_d(q^*q).
	\]
	Since $L_{r_n}(q^*q)$ is nonnegative, also $L_d(q^*q)$ will not be negative, and applying Theorem \ref{moments_problem_unit_circle} again, we obtain that $(d_k)_k$ represents the moments of a measure $\mu$ on $\f T$, that is also a probability measure since $d_0=1$.
	\[
	\lim_{n\to \infty}\frac{1}{n} \sum_{i=1}^n \omega_{i,n}^k = d_k = \int_{\f T} z^k d\mu(z), \qquad \forall k\in \f Z
	\]
	Stone-Weierstrass Theorem \ref{Weierstrass} let us approximate any continuous function on $\f T$ with a Laurent polynomial, so
	\[
	\lim_{n\to \infty}\frac{1}{n} \sum_{i=1}^n G(\omega_{i,n}) = \int_{\f T} G(z) d\mu(z)
	\]
	holds for every $G\in C_c(\f C)$. Going back to the eigenvalues $\lambda_{i,n}$, suppose that $F$ is a continuous function defined on $c\f T$, and $G(x):= F(cx)$. 
	\[
	\lim_{n\to \infty}\frac{1}{n} \sum_{i=1}^n F(\lambda_{i,n}) =  \lim_{n\to \infty}\frac{1}{n} \sum_{i=1}^n G(\omega_{i,n}) = \int_{\f T} G(z) d\mu(z) = \int_{\f T} F(cz) d\mu(z) =  \int_{c\f T} F(z) d\mu\left(\frac zc\right)
	\]
	The last relation let us conclude that $\serie A$ admits a spectral measure, and consequently admits also a spectral symbol.	
\end{proof}

\section{Curves and Weil Systems}

Given a curve $X$ of genus $g$ on a finite field $\f F_q$, we can define its Zeta function as 
\[
Z(t) = \frac{P(t)}{(1-t)(1-qt)}
\]
where $P(t)$ is a polynomial with the following properties (Weil Conjecture):
\begin{itemize}
	\item $P(t)\in\f Z[t]$,
	\item $P(t)=\prod_{i=1}^{2g}(1-\lambda_{i}t) $ where $|\lambda_i|=\sqrt q$ and if $\lambda_i\in \f R$ it has even multiplicity,
	\item There exists natural numbers $B_m$ such that
	\[
	P(t) = (1-t)(1-qt)\prod_{m=1}^\infty (1-t^m)^{-B_m},
	\]
	\item There exist natural numbers $N_m$ such that
	\[
	P(t) = (1-t)(1-qt)\exp\left( \sum_{i=1}^n N_n\frac {t^n}n \right),
	\]
	\item $P(t)$ is the characteristic polynomial of the Frobenius endomorphism induced to the Tate module (or the first \'Etale cohomology group) of the curve.
\end{itemize}
Given this properties one can prove relations between $\lambda_i$, $N_m$ and $B_m$. Namely
\[
N_m = q^m +1 - \sum_{i=1}^{2g} \lambda_i^m = \sum_{d|m} dB_d, 
\qquad
 mB_m = \sum_{d|m} \mu\left(\frac md \right) \left(q^d+1-\sum_{i=1}^{2g} \lambda_i^d\right) = \sum_{d|m}\mu\left(\frac md \right) N_d.
\]
An \textit{asymptotically exact family of Weil systems} is a sequence of curves $\serie X$, where $X_n$ has genus $n$, and such that
\[
\beta_m = \lim_{n\to \infty} \frac{B_m(X_n)}{n}\in \f R
\]
exists for every $m\in \f N$. Thanks to the relations above, we can see that  this condition is equivalent to require that 
\[
\nu_m = \lim_{n\to \infty} \frac{\sum_{i=1}^{2n} \lambda_{i,n}^m}{n}\in \f R
\]
for every $m\in \f N$, where $\lambda_{i,n}$ are the roots of the polynomials $P_n(t)$ associated to $X_n$.

 Here we have all the hypothesis to apply Lemma \ref{spectral_symbol_unit_circle}, with the exception for $A_n\in\f C^{n\times n}$. In fact, $\lambda_{i,n}$ are eigenvalues of the Frobenius endomorphisms, that has elements in the $p$-adic field $\f Q_p$. We can notice, though, that the result does not depend on the matrices $A_n$ themselves, but only on their eigenvalues, and we can conclude nonetheless that there exists a limit measure 
\[
\mu = \lim_{n\to\infty}\mu_n = \lim_{n\to\infty}\frac 1n \sum_{i=1}^{n} \delta_{\lambda_{i,n}}.
\]
In the case of asymptotically exact families, it is possible to give an estimate of $\nu_k$, that follows from the proved bound on $\beta_m$ 
\[
\sum_{m=1}^{\infty} \frac{m\beta_m}{q^{m/2}+1}\le 1.
\] 
In fact, using
\[
-\nu_m= -\lim_{n\to\infty}\frac{\sum_{i=1}^{2g} \lambda_{i,n}^m}{n}  = \lim_{n\to\infty}\frac{N_m(X_n)}{n} =
\lim_{n\to\infty}\sum_{d|m} \frac{dB_d(X_n)}{n} =
\sum_{d|m} d\beta_d\ge 0 
\]
it is easy to prove that
\[  
0\le -\sum_{m=1}^{\infty} \nu_m q^{-m/2} = \sum_{m=1}^{\infty} \frac{m\beta_m}{q^{m/2}-1}\le 1.
\]
Owing to Lemma \ref{abs_continuous_spectral_measure}, we can conclude that the spectral measure $\mu$ has also a density $M(x)$ with respect to the Lebesgue measure.
These results were already proven in \cite{Tsa} and \cite{Tsa2}.

\subsection{Food for Thought}

One can notice that we used only particular cases of the results we proved in Section 2, so one can ask if a generalization of ideas could lead to something more.

For example, consider $\lambda_{i,n}$ complex numbers of the same norm $c$, with $1\le i\le 2n$. It's fairly easy to see that, in analogy with the curves theory, the following assertions hold.
\begin{itemize}
	\item We can define $N_m,B_m$ with
	\[
	N_m := a^{m} + 1 -\sum_{i=1}^{2n}\lambda_{i,n}^m, \qquad mB_m = \sum_{d|m}\mu\left(\frac md \right) N_d
	\]
	where $a\in \f N$.
	\item If $P_n(x) = \prod_{i=1}^{2n} (1-\lambda_{i,n}x)$, then
	\[
	P_n(x) = (1-ax)(1-x)\exp\left( \sum_{k=1}^\infty  N_k\frac {x^k}k \right)
	= (1-ax)(1-x)\prod_{k=1}^\infty (1-x^m)^{-B_m}
	\]
\end{itemize}
 If we require also that the averages of the powers of $\lambda_{i,n}$ converge, then we can apply 
 Lemma \ref{spectral_symbol_unit_circle} 
and conclude that there exists an asymptotic measure for $\lambda_{i,n}$. 

In general $N_m,B_m$ will be complex numbers, so it does not seem an interesting generalization.
We thus consider an ulterior hypothesis that brings the setting closer to Weil systems. If we suppose that $P_n(x) \in \f Z[x]$, then we can prove that $N_m$ are integer numbers using properties of symmetric functions, and therefore $B_m$ are rational numbers. 
The set of asymptotic measures that describe the polynomials $P_n(x)$ strictly contains the measures describing  asymptotically exact family of Weil systems, and in particular, it includes some atomic measures. 

It may thus be interesting to study the exact Weil system through their asymptotic measures by looking at it from a more general context, or considering the spectral symbols associated to the measures.

Eventually, we report an example of $\lambda_{i,n}$ for which $N_m,B_m$ are integers, but the asymptotic measure is yet atomic.

\begin{example}
	Let $\lambda_{i,n} = \sqrt{k} (-1)^i$, where $k$ is a natural number different from 0, and let $a$ be an integer number. In this case, it is easy to see that the asymptotic measure $\mu$ on the circle with ray $\sqrt{k}$ is atomic $\mu = (\delta_{\sqrt{k}} + \delta_{-\sqrt{k}})/2$.

	The polynomial $P_n(x)$ has integer coefficients, since
	\[
	P_n(x) = \prod_{i=1}^n (1-\sqrt{k}x)(1+\sqrt{k}x)
	=  (1-kx^2)^n
	\]
	and thus $N_m$ are integers
	\[
	N_m = \begin{cases}
	a^m + 1 & m \text{ odd,}\\
	a^m +1 -2nk^{m/2} & m \text{ even.}
	\end{cases}
	\]
	Exploiting the properties of Moebius function, we can also prove that $B_m$ is integer, in fact 
	\[
	\sum_{d|m}\mu\left(\frac md \right) N_d \equiv 0\pmod m\qquad \forall m\in \f N.
	\]
	 As one can see from the relations above, $N_m$ is negative when $m$ is even and $n$ is large enough. In fact, $\nu_n$ are all non-negative integers and
	\[
-\sum_{m=1}^{\infty} \nu_m c^{-m}	=\sum_{m=1}^{\infty} \frac{m\beta_m}{c^m -1} = -\infty.
	\]
\end{example}

\end{document}